\documentclass[final]{amsart}
\pdfoutput=1
\usepackage[utf8]{inputenc}
\usepackage[T1]{fontenc}
\usepackage[english]{babel}
\usepackage{amssymb, amsmath, amsthm, tikz-cd, cite, tikz, pgfplots}
\numberwithin{equation}{section}
\pgfplotsset{compat=1.14}

\usepackage[hyperfootnotes=false,hypertexnames=false,final]{hyperref}

\theoremstyle{plain}
\newtheorem{thm}{Theorem}[section]
\newtheorem{lm}[thm]{Lemma}

\newtheorem{prop}[thm]{Proposition}
\newtheorem{conj}[thm]{Conjecture}
\newtheorem{prb}[thm]{Problem}

\newtheorem*{main}{Main Theorem}

\theoremstyle{definition}

\newcommand{\QQ}{\mathbb{Q}}
\newcommand{\RR}{\mathbb{R}}
\newcommand{\CC}{\mathbb{C}}

\newcommand{\A}{\mathcal{A}}
\newcommand{\B}{\mathcal{B}}
\newcommand{\C}{\mathcal{C}}
\newcommand{\e}{\mathrm{e}}
\newcommand{\I}{\mathrm{i}}

\DeclareMathOperator{\tr}{tr}

\makeatletter
\def\iddots{\mathinner{\mkern1mu\raise\p@
    \vbox{\kern7\p@\hbox{.}}\mkern2mu
    \raise4\p@\hbox{.}\mkern2mu\raise7\p@\hbox{.}\mkern1mu}}
\makeatother

\title[Minkowski sums of a twisted cubic segment]{Semi-algebraic properties of Minkowski sums of a twisted cubic segment}
\author{Arthur Bik}
\address{Universit\"at Bern, Mathematisches Institut, Alpeneggstrasse~22, 3012 Bern, Switzerland}
\email{arthur.bik@math.unibe.ch}

\author{Adam Czapli\'nski}
\address{Universit\"at Siegen, Department Mathematik, Walter-Flex-Stra{\ss}e~3, 57068 Siegen, Germany}
\email{adam.czaplinski@uni-siegen.de}

\author{Markus Wageringel}
\address{Universit\"at Osnabr\"uck, Institut f\"ur Mathematik, Albrechtstraße~28\,A, 49076 Osnabr\"uck, Germany}
\email{markus.wageringel@uos.de}

\makeatletter
\hypersetup{pdftitle={\@title},%
  pdfauthor={Arthur Bik, Adam Czapli\'nski, Markus Wageringel}}
\makeatother

\thanks{The first author was supported by the Vici grant \emph{Stabilisation in Algebra and Geometry} from the Netherlands Organisation for Scientific Research (NWO). The third author was supported by DFG grant GK 1916 \emph{Kombinatorische Strukturen in der Geometrie} of the German Research Foundation.}


\begin{document}

\begin{abstract}
We find a semi-algebraic description of the Minkowski sum $\A_{3,n}$ of $n$ copies of the twisted cubic segment $\{(t,t^2,t^3)\mid -1\leq t\leq 1\}$ for each integer~$n\geq3$. These descriptions provide efficient membership tests for the sets~$\A_{3,n}$. These membership tests in turn can be used to resolve some instances of the underdetermined matrix moment problem, which was formulated by Michael Rubinstein and Peter Sarnak in order to study problems related to $L$-functions and their zeros.
\end{abstract}
\maketitle

\section{Introduction}

The zeros of $L$-functions are known to be able to describe various geometrical and arithmetical objects and are the subjects of several conjectures (cf. \cite{Newmansconj15,Sixshortchap,ZerodistrDir16}). For example, the Generalized Riemann Hypothesis conjectures that all non-trivial zeros of an $L$-function have real part~$\frac{1}{2}$ and the Grand Simplicity Hypothesis asserts that the imaginary parts of zeros of Dirichlet $L$-functions are linearly independent over~$\QQ$ (cf. \cite{Mi09}). $L$-functions can also be encountered in proofs of the Prime Number Theorem (cf. \cite{UAnzdPrimRiem}) and in primality tests (cf. \cite{RHaTfP76}).\bigskip

Let $\rho$ be an automorphic cusp form and let $L(s,\rho)$ be its standard $L$-function. A conjecture, which has been verified in many cases, states that under certain conditions the function $L(s,\rho)$ has an
analytic continuation $\Lambda(s,\rho)$ that satisfies the functional equation
\[
\Lambda(1-s,\rho)=W(\rho)N_\rho^{s-1/2}\Lambda(s,\rho),
\]
where $N_\rho$ is the conductor of $\rho$ and $W(\rho)$ is either $1$ or $-1$. The sign $W(\rho)\in\{\pm1\}$ is called the root number of $\rho$. The problems of computing root numbers and counting the zeros of $L$-functions are related to the following problem with $\log N_\rho\approx n$.

\begin{prb}[The underdetermined matrix moment problem]\label{prb:undermat}
Determine the possible sets of eigenvalues of a real orthogonal $(2n+1)\times(2n+1)$ matrix $A$ given its first $k\leq n$ moments $\tr(A),\tr(A^2),\dots,\tr(A^k)$.
\end{prb}

This problem is the object of study in the paper~\cite{RuSa_prep} by Michael Rubinstein and Peter Sarnak. For the full background and relevance of the problem, we refer to this paper.\bigskip

Let $A$ be a real orthogonal $(2n+1)\times(2n+1)$ matrix. Then its eigenvalues are
\[
\det(A),\e^{\I\theta_1},\e^{-\I\theta_1},\dots,\e^{\I\theta_n},\e^{-\I\theta_n}
\]
for some $\theta_1,\dots,\theta_n\in[0,\pi]$. And conversely, any such sequence is the spectrum of a real orthogonal $(2n+1)\times(2n+1)$ matrix. We have 
\[
\tr(A^j)=\det(A)^j+2\sum_{i=1}^n\cos(j\theta_i)
\]
and $\cos(j\theta_i)=T_j(\cos(\theta_i))$ for all integers $j\geq 1$, where $T_j$ is the $j$-th Chebyshev polynomial of the first kind. The polynomial $T_j(x)$ has degree $j$. So, given $\det(A)$ and $\tr(A),\tr(A^2),\dots,\tr(A^k)$ for some integer $k\leq n$, we can compute $\sum_{i=1}^n\cos(\theta_i)^j$ for each $j\in\{1,\dots,k\}$ using Gaussian elimination on the coefficient vectors of $T_1,\dots,T_k$. As $\det(A)\in\{\pm1\}$ only has finitely many possible values, we write $t_i=\cos(\theta_i)\in[-1,1]$ and see that Problem~\ref{prb:undermat} reduces to the following problem.

\begin{prb}[The moment curve problem]\label{prb:momcurve}
Determine the set
\[
\left\{(t_1,\dots,t_n)\in[-1,1]^n\,\middle|\,\sum_{i=1}^nt_i^j=x_j\mathrm{~for~all~}j\in\{1,\dots,k\}\right\}
\]
given the real numbers $x_1,\dots,x_k\in\RR$.
\end{prb}

This problem was also formulated by Michael Rubinstein and Peter Sarnak. Note that given the first $k$ power sums of $t_1,\dots,t_n$, we can compute all symmetric polynomial expressions in $t_1,\dots,t_n$ of degree at most $k$. So if $k=n$, then we are able to compute the coefficients of the polynomial $(x-t_1)\cdots(x-t_n)$, which not only allows us to recover $t_1,\dots,t_n$, but also shows that $t_1,\dots,t_n\in\CC$ are unique up to reordering. So we are interested in the case where $k<n$. In this case, Michael Rubinstein and Peter Sarnak propose the following strategy: consider the set $\C_k:=\{(t,t^2,\dots,t^k)\mid -1\leq t\leq 1\}\subseteq\RR^k$ and define 
\[
\A_{k,n}:=\underbrace{\C_k+\C_k+\dots+\C_k}_{n}
\]
to be the Minkowski sum of $n$ copies of $\C_k$ for each integer $n\geq 1$. Then we can determine the set of tuples $(t_1,\dots,t_n)\in[-1,1]^n$ such that 
\[
\sum_{i=1}^nt_i^j=x_j
\]
for all $j\in\{1,\dots,k\}$ recursively by first computing the set of $t_n\in[-1,1]$ such that
\[
(x_1,x_2,\dots,x_k)-(t_n,t_n^2,\dots,t_n^k)\in\A_{k,n-1}.
\]
In order to do the latter, we need an efficient membership test for the set $\A_{k,n}$ for all $n>k$. For $k\in\{1,2\}$, this is easy. In general, one way to get an efficient membership test would be to describe the sets $\A_{k,n}$ implicitly using only equalities and inequalities involving polynomial expressions in $x_1,\dots,x_k$ and unions---in other words, using semi-algebraic descriptions of the sets $\A_{k,n}$. In this paper, we provide exactly such descriptions in the case that $k=3$.

\section{Main results}\label{sect:mainres}

Let $n\geq3$ be a positive integer. Our first result describes the boundary of~$\A_{3,n}$. We need this result in order to prove the Main Theorem. However, it also provides us with a piecewise parametrization, which is useful for rendering a visualization of~$\A_{3,n}$. See Figure~\ref{fig:a34} for an example. 

\begin{figure}[ht]
  \includegraphics[width=\textwidth]{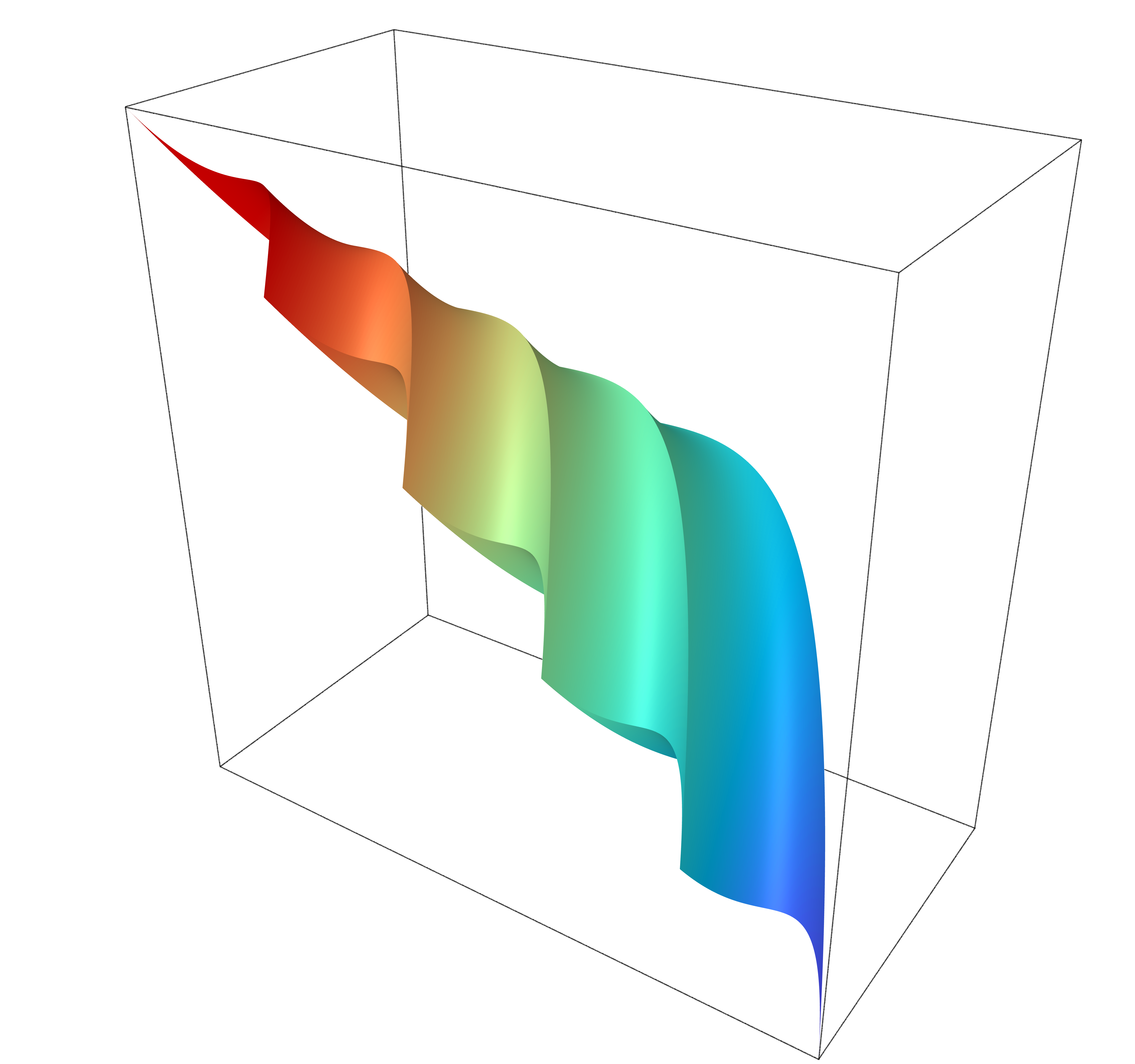}
  \put(-46,71){$(-5,0,-5)$}
  \put(-355,304){$(5,5,5)$}
  \put(-205,41){$x$}
  \put(-248,114){$y$}
  \put(-298,196){$z$}
  \caption{A rendering of the semi-algebraic set $\A_{3,5}$. Interactive 3D models of $\A_{3,n}$ are available at \url{https://mathsites.unibe.ch/bik/A3n.html} for $n=1,\dots,20$.}
  \label{fig:a34}
\end{figure}

Before we give the semi-algebraic description of $\A_{3,n}$, we first discuss the intuition behind it. As Figure~\ref{fig:a34} for $n=5$ and the interactive 3D models for $n=3,\dots,20$ demonstrate, the set $\A_{3,n}$ looks like an oyster with an upper and lower shell forming the boundary. We call these upper and lower shells $\B_n^+$ and $\B_n^-$ respectively. These two shells have identical projections to the $(x,y)$-plane, which we denote by $\B^{\flat}_n$, and both projection maps are one-to-one. This yields a first description of $\A_{3,n}$: for a point $(x,y,z)\in\RR^3$ to lie in $\A_{3,n}$, it is necessary that $(x,y)$ lies in $\B^{\flat}_n$. When this is the case, the point lies in $\A_{3,n}$ precisely when it lies below $\B_n^+$ and above~$\B_n^-$.\bigskip

Each of the two shells consists of $n-1$ spiraling ridges. These are the sets $\C_{k,n-k-1}^+$ and $\C_{\ell,n-\ell-1}^-$ defined below respectively. The projections of the ridges are easily visualized on $\B^{\flat}_n$. See Figure~\ref{fig:projection+}. We can now reformulate our first description of $\A_{3,n}$ in the following way: let $(x,y,z)\in\RR^3$ be a point such that $(x,y)$ lies in $\B^{\flat}_n$. Then $(x,y)$ lies in the projections of $\C_{k,n-k-1}^+$ and $\C_{\ell,n-\ell-1}^-$ for some $k,\ell$. The point $(x,y,z)$ lies in $\A_{3,n}$ precisely when is lies below $\C_{k,n-k-1}^+$ and above $\C_{\ell,n-\ell-1}^-$.\bigskip

Next, we think of $(x,y)$ as being fixed. This turns the conditions of lying below $\C_{k,n-k-1}^+$ and above $\C_{\ell,n-\ell-1}^-$ into conditions on $z$. The former condition is equivalent to $z$ being at most the biggest root of some parabola $f=Az^2+Bz+C$ with $A>0$ and can thus be expressed as $f(z)$ being at most $0$ or $z$ is at most the value $-B/2A$ where $f$ attains its minimum. See Figure~\ref{fig:fk1}. Similarly, the latter condition is equivalent to $z$ being at least the smallest root of a parabola $f=Az^2+Bz+C$ with $A>0$ and can be expressed as $f(z)$ being at most $0$ or $z$ is at least the value $-B/2A$ where $f$ attains its minimum. This is our description of $\A_{3,n}$.\bigskip

In order to state our results precisely, we define the following sets:
\begin{itemize}
\item for all integers $k,\ell\geq1$ and $a,b\geq0$, we take
\begin{align*}
\C^{+}_{k,a}&=\left\{k\begin{pmatrix}s\\s^2\\s^3\end{pmatrix}+\begin{pmatrix}t\\t^2\\t^3\end{pmatrix}\,\middle|\,-1\leq s\leq t\leq 1\right\}+a\begin{pmatrix}1\\1\\1\end{pmatrix},\\
\C^{-}_{\ell,b}&=\left\{\begin{pmatrix}s\\s^2\\s^3\end{pmatrix}+\ell\begin{pmatrix}t\\t^2\\t^3\end{pmatrix}\,\middle|\,-1\leq s\leq t\leq 1\right\}+b\begin{pmatrix}-1\\1\\-1\end{pmatrix},
\end{align*}
\item we take $\B^{+}_{n}=\bigcup_{k=1}^{n-1}\C^{+}_{k,n-k-1}$ and $\B^{-}_{n}=\bigcup_{\ell=1}^{n-1}\C^{-}_{\ell,n-\ell-1}$ and
\item we let $\B^{\flat}_{n}$ be the set consisting of all points $(x,y)\in\RR^2$ such that $ny\geq x^2$ and $y\leq n-1+(x+2i-(n-1))^2$ for each $i\in\{0,\dots,n-1\}$.
\end{itemize}
We also let $\pi\colon\RR^3\to\RR^2$ be the projection map sending $(x,y,z)\mapsto(x,y)$.

\begin{thm}\label{thm:boundaryA3n}
Let $(x,y)\in\B^{\flat}_{n}$ be a point and $z\in\RR$ be a real number. 
\begin{itemize}
\item[(a)] The boundary of $\A_{3,n}$ is the union of $\B^{+}_{n}$ and $\B^{-}_{n}$.
\item[(b)] We have $\pi(\B^{+}_{n})=\pi(\B^{-}_{n})=\B^{\flat}_{n}$.
\item[(c)] There exist unique numbers $z^{+},z^{-}\in\RR$ such that $(x,y,z^{+})\in\B^{+}_{n}$ and $(x,y,z^{-})\in\B^{-}_{n}$. We have $z^{+}\geq z^{-}$. Moreover, equality holds precisely when the point $(x,y)$ lies on the boundary of $\B^{\flat}_{n}$.
\item[(d)] We have $(x,y,z)\in\A_{3,n}$ if and only if $z^{+}\geq z\geq z^{-}$.
\item[(e)] Every point on the boundary of $\A_{3,n}$ can be written as 
\[
p=\begin{pmatrix}t_1\\t_1^2\\t_1^3\end{pmatrix}+\dots+\begin{pmatrix}t_n\\t_n^2\\t_n^3\end{pmatrix}
\]
for some tuple $(t_1,\dots,t_n)\in[-1,1]^n$. The set $\{t_1,\dots,t_n\}\setminus\{-1,1\}$ has at most two elements and the tuple $(t_1,\dots,t_n)$ is unique up to permutation of its entries.
\end{itemize}
\end{thm}

In Section~\ref{sect:kC+ellC}, we find semi-algebraic descriptions of (in particular) $\C^+_{k,a}$ and $\C^{-}_{\ell,b}$. To write these descriptions down, we define 
\begin{align*}
A_{k\ell} &= k\ell(k+\ell)^2,\\
B_{k\ell}(x,y) &= 2k\ell x(2x^2 - 3(k+\ell)y),\\
C_{k\ell}(x,y) &= x^6 - 3(k+\ell)x^4y + 3(k^2+k\ell+\ell^2)x^2y^2 - (k-\ell)^2(k+\ell)y^3,\\
D_{k\ell}(x,y) &= (k+\ell)y-x^2,\\
f_{k\ell}(x,y,z) &= A_{k\ell}z^2+B_{k\ell}(x,y)z+C_{k\ell}(x,y)
\end{align*}
for all positive integers $k,\ell\geq1$. Note here that 
\[
B_{k\ell}^2-4A_{k\ell}C_{k\ell}=4k\ell(\ell-k)^2D_{k\ell}^3
\]
for all $k,\ell$. We then use these descriptions together with the previous theorem to prove our main result. Take the following sets:
\begin{align*}
X&=\bigcup_{k=1}^{n-1}\left\{\begin{pmatrix}x\\y\\z\end{pmatrix}+(n-k-1)\begin{pmatrix}1\\1\\1\end{pmatrix}\in\RR^3\,\middle|\,\begin{array}{l}y\leq k+(x+k)^2,\\y\geq (k+1)^{-1}x^2,\\y\leq 1+k^{-1}(x-1)^2 \mbox{ and}\\z\leq \frac{-B_{k1}(x,y)}{2A_{k1}}\mbox{ or }f_{k1}(x,y,z)\leq0\end{array}\right\},\\
Y&=\bigcup_{\ell=1}^{n-1}\left\{\begin{pmatrix}x\\y\\z\end{pmatrix}+(n-\ell-1)\begin{pmatrix}-1\\1\\-1\end{pmatrix}\in\RR^3\,\middle|\,\begin{array}{l}y\leq \ell+(x-\ell)^2,\\y\geq (\ell+1)^{-1}x^2,\\y\leq 1+ \ell^{-1}(x+1)^2\mbox{ and}\\z\geq \frac{-B_{1\ell}(x,y)}{2A_{1\ell}}\mbox{ or }f_{1\ell}(x,y,z)\leq0\end{array}\right\}.
\end{align*}

\begin{main}
We have $\A_{3,n}=X\cap Y$.
\end{main}

\subsection*{Structure of the paper}
The first step of the proof is to show that the boundary of $\A_{3,n}$ is contained in the union of $\B^{+}_{n}$ and $\B^{-}_{n}$. We start doing this by proving a result about representations of points on the boundary of $\A_{3,n}$ in Section~\ref{sect:2rep}. In Section~\ref{sect:4vsn}, we prove the statement for $n=4$ and then conclude that it holds for all~$n\geq3$. After that, in Section~\ref{sect:kC+ellC}, we find semi-algebraic descriptions for the components that make up the boundary. And then, in Section~\ref{sect:B+-}, we study the sets $\B^{+}_{n}$ and $\B^{-}_{n}$ in more detail and prove Theorem~\ref{thm:boundaryA3n} and the Main Theorem. We conclude the paper by discussing the problem for higher dimensions in Section~\ref{sect:future}.

\subsection*{Acknowledgments}
The problem that this paper solves was brought to our attention by Bernd Sturmfels during the graduate student meeting on applied algebra and combinatorics held in Leipzig on 18--20 February 2019. We would like to thank him for doing so and we would like to thank the organizers of this meeting for making it possible. We would also like to thank Peter Sarnak for explaining the origin and relevance of the problem to us. Lastly, we would like to thank the anonymous referees for their precise reading of our paper and their helpful comments.

\section{Representations of points on the boundary of \texorpdfstring{$\A_{3,n}$}{A3n}}\label{sect:2rep}

The goal of this and the next section is to prove that the boundary of $\A_{3,n}$ is contained in the union of $\B^{+}_{n}$ and $\B^{-}_{n}$. We start with the following proposition.

\begin{prop}\label{prop_2par}
Let $p\in\RR^3$ be a point on the boundary of $\A_{3,n}$ and write
\[
p=\begin{pmatrix}t_1\\t_1^2\\t_1^3\end{pmatrix}+\dots+\begin{pmatrix}t_n\\t_n^2\\t_n^3\end{pmatrix}
\]
for some tuple $(t_1,\dots,t_n)\in[-1,1]^n$. Then the set $\{t_1,\dots,t_n\}\setminus\{-1,1\}$ has at most two elements.
\end{prop}
\begin{proof}
Fix indices $1\leq i<j<k\leq n$ and consider the map
\begin{align*}
\varphi\colon[-1,1]^3&\to\RR^3,\\
(r,s,t)&\mapsto\begin{pmatrix}r\\r^2\\r^3\end{pmatrix}+\begin{pmatrix}s\\s^2\\s^3\end{pmatrix}+\begin{pmatrix}t\\t^2\\t^3\end{pmatrix}+\sum_{\substack{\ell\in\{1,\dots,n\}\\\ell\not\in\{i,j,k\}}}\begin{pmatrix}t_\ell\\t_\ell^2\\t_\ell^3\end{pmatrix}.
\end{align*}
Then we have $p=\varphi(t_i,t_j,t_k)$. The Jacobian of $\varphi$ at the point $(t_i,t_j,t_k)$ is
\[
\begin{pmatrix}1&1&1\\2t_i&2t_j&2t_k\\3t_i^2&3t_j^2&3t_k^2\end{pmatrix}=\begin{pmatrix}1\\&2\\&&3\end{pmatrix} \begin{pmatrix}1&1&1\\t_i&t_j&t_k\\t_i^2&t_j^2&t_k^2\end{pmatrix}
\]
and hence has rank $3$ if $t_i\neq t_j\neq t_k\neq t_i$. So, if in addition $t_i,t_j,t_k\not\in\{-1,1\}$, then every point in a small neighborhood around $p$ is in the image of $\varphi$ by the inverse function theorem. As this cannot happen for a point $p$ on the boundary of $\A_{3,n}$, it follows that the set $\{t_1,\dots,t_n\}\setminus\{-1,1\}$ has at most two elements.
\end{proof}

From the proposition, it immediately follows that the boundary of $\A_{3,n}$ is contained in the union of the sets
\[
\left\{k\begin{pmatrix}s\\s^2\\s^3\end{pmatrix}+\ell\begin{pmatrix}t\\t^2\\t^3\end{pmatrix}\,\middle|\,-1\leq s\leq t\leq1\right\}+a\begin{pmatrix}1\\1\\1\end{pmatrix}+b\begin{pmatrix}-1\\1\\-1\end{pmatrix}
\]
over all integers $k,\ell\geq1$ and $a,b\geq0$ such that $k+\ell+a+b=n$. So to prove that the boundary of $\A_{3,n}$ is contained in the union of $\B^{+}_{n}$ and $\B^{-}_{n}$, it suffices to prove that every point that is contained in one of these sets and is not contained in $\B^{+}_{n}\cup\B^{-}_{n}$ is also not contained in the boundary of $\A_{3,n}$.

\begin{lm}\label{lm:elim}
Let $k\geq k'\geq1$, $\ell\geq\ell'\geq1$, $a\geq a'\geq0$ and $b\geq b'\geq0$ be integers. Take $n=k+\ell+a+b$, $n'=k'+\ell'+a'+b'$ and $-1<s<t<1$. Assume that 
\[
p'=k'\begin{pmatrix}s\\s^2\\s^3\end{pmatrix}+\ell'\begin{pmatrix}t\\t^2\\t^3\end{pmatrix}+a'\begin{pmatrix}1\\1\\1\end{pmatrix}+b'\begin{pmatrix}-1\\1\\-1\end{pmatrix}
\]
is not contained in the boundary of $\A_{3,n'}$. Then
\[
p=k\begin{pmatrix}s\\s^2\\s^3\end{pmatrix}+\ell\begin{pmatrix}t\\t^2\\t^3\end{pmatrix}+a\begin{pmatrix}1\\1\\1\end{pmatrix}+b\begin{pmatrix}-1\\1\\-1\end{pmatrix}
\]
is not contained in the boundary of $\A_{3,n}$.
\end{lm} 
\begin{proof}
If the point $p'$ does not lie on the boundary of $\A_{3,n'}$, then the point $p$ cannot lie on the boundary of
\[
\A_{3,n'}+(k-k')(s,s^2,s^3)+(\ell-\ell')(t,t^2,t^3)+(a-a')(1,1,1)+(b-b')(-1,1,-1)
\]
and hence it can also not lie on the boundary of $\A_{3,n}$. 
\end{proof}

To prove that the boundary of $\A_{3,n}$ is contained in the union of $\B^{+}_{n}$ and $\B^{-}_{n}$, we need to eliminate the cases where one of the following conditions holds:
\begin{enumerate}
\item $k,\ell\geq2$,
\item $k=\ell=1$ and $a,b>0$,
\item $k=1$, $\ell>1$ and $a>0$,
\item $k>1$, $\ell=1$ and $b>0$.
\end{enumerate}
Using the previous lemma, we see that it suffices to eliminate the cases where 
\[
(k,\ell,a,b)\in\{(1,1,1,1),(1,2,1,0),(2,1,0,1),(2,2,0,0)\}
\]
and hence we first consider the boundary of $\A_{3,n}$ for $n=4$.

\section{The boundary of \texorpdfstring{$\A_{3,4}$}{A34} versus the boundary of \texorpdfstring{$\A_{3,n}$}{A3n}}\label{sect:4vsn}

In this section, we show that the boundary of $\A_{3,n}$ is contained in $\B^{+}_{n}\cup\B^{-}_{n}$ for $n=4$ and conclude from this that the same statement holds for all $n\geq3$. We need to show that no point of
\[
\left\{k\begin{pmatrix}s\\s^2\\s^3\end{pmatrix}+\ell\begin{pmatrix}t\\t^2\\t^3\end{pmatrix}\,\middle|\,-1< s< t<1\right\}+a\begin{pmatrix}1\\1\\1\end{pmatrix}+b\begin{pmatrix}-1\\1\\-1\end{pmatrix}
\]
is contained in the boundary of $\A_{3,4}$ for $(k,\ell,a,b)=(1,1,1,1),(1,2,1,0)$, $(2,1,0,1)$, $(2,2,0,0)$.
We start with the case $(k,\ell,a,b)=(2,2,0,0)$.

\begin{prop}\label{prop:2C+2C}
Take $-1<s<t<1$. Then the point
\[
p=2(s,s^2,s^3)+2(t,t^2,t^3)
\] 
does not lie on the boundary of $\A_{3,4}$.
\end{prop}
\begin{proof}
Consider the system of equations
\begin{align*}
2s+2t&=t_1+t_2+t_3+t_4,\\
2s^2+2t^2&=t_1^2+t_2^2+t_3^2+t_4^2,\\
2s^3+2t^3&=t_1^3+t_2^3+t_3^3+t_4^3
\end{align*}
with the additional conditions that $-1<t_1,t_2,t_3,t_4<1$ are pairwise distinct. If this system has a solution that satisfies the additional conditions, then the point~$p$ cannot lie on the boundary of $\A_{3,4}$ by Proposition~\ref{prop_2par}. It turns out that such a solution $(t_1,t_2,t_3,t_4)$ can even be found when we assume that $t_1+t_2=t_3+t_4$. Indeed, let $0\neq\alpha\neq\beta\neq0$ be such that $|\alpha|,|\beta|<1-\frac{1}{2}|s+t|$ and $\alpha^2+\beta^2=\frac{1}{2}(s-t)^2$. Then 
\[
(t_1,t_2,t_3,t_4)=\left(\frac{s+t}{2}+\alpha,\frac{s+t}{2}-\alpha,\frac{s+t}{2}+\beta,\frac{s+t}{2}-\beta\right)
\]
is a solution to the system equalities so that $-1<t_1,t_2,t_3,t_4<1$ are pairwise distinct. One can check that $|\alpha|,|\beta|<1-\frac{1}{2}|s+t|$ and $\alpha^2+\beta^2=\frac{1}{2}(s-t)^2$ for $\alpha,\beta=\pm\frac{1}{2}(s-t)$. Here we use that $|s-t|+|s+t|\leq2\cdot\max(|s|,|t|)<2$. It follows that for any point $(\alpha,\beta)$ on the circle given by
\[
\alpha^2+\beta^2=\frac{1}{2}(s-t)^2
\]
that is sufficiently close to $(\frac{1}{2}(s-t),\frac{1}{2}(s-t))$ also satisfies these conditions. So to conclude the proof, we simply let $(\alpha,\beta)$ be such a point with $0\neq\alpha\neq\beta\neq0$.
\end{proof}

Next, we take care of the case $(k,\ell,a,b)=(1,1,1,1)$.

\begin{lm}\label{lm_1110down}
Take $\delta>0$. Then there exists an $\varepsilon>0$ such that 
\[
(s,s^2,s^3)+(t,t^2,t^3)+(1,1,1)-(0,0,\varepsilon')\in\A_{3,3}
\]
for all $-1<s<t<1$ and $0\leq\varepsilon'\leq\varepsilon$ with $1-t,t-s,s-(-1)\geq\delta$.
\end{lm}
\begin{proof}
For $0\leq\lambda\ll1$, consider $(t_1,t_2,t_3)=(s+\mu,t+\lambda-\mu,1-\lambda)$ where 
\begin{align*}
\mu&=(\lambda+(t-s)-\sqrt{d})/2,\\
d&=(t-s)^2+2(t-s)\lambda+4(1-t)\lambda-3\lambda^2.
\end{align*}
Assuming that $d\geq0$ and $-1\leq t_1,t_2,t_3\leq1$, one can check that
\[
(s,s^2,s^3)+(t,t^2,t^3)+(1,1,1)-(0,0,\varepsilon'(\lambda))=\sum_{i=1}^3(t_i,t_i^2,t_i^3)\in\A_{3,3}
\]
for $\varepsilon'(\lambda)=3\lambda(1-t-\lambda)(1-s-\lambda)$. Note that $\varepsilon'(0)=0$. So to prove the lemma, it suffices to find a $c>0$ such that $d\geq0$ and $-1\leq t_1,t_2,t_3\leq1$ for all $0\leq\lambda\leq c$ and $\varepsilon'(c)>0$, because we can then take $\varepsilon=\varepsilon'(c)$ and have $[0,\varepsilon]\subseteq\varepsilon'([0,c])$.\bigskip

Take $c=\min(\delta,\delta^2)/1000>0$ and assume that $0\leq\lambda\leq c$. Then we have
\[
d\geq (t-s)^2+3\lambda(2\delta-\lambda)\geq (t-s)^2\geq0
\]
and hence $\sqrt{d}\geq t-s$. We also have
\[
\sqrt{d}-(t-s)=\frac{d-(t-s)^2}{\sqrt{d}+(t-s)}\leq\frac{2(t-s)\lambda+4(1-t)\lambda-3\lambda^2}{2(t-s)}\leq\frac{12\lambda}{2\delta}\leq \delta/100
\]
since $t-s,1-t\leq 2$. Hence $|\mu|\leq \delta/100$ and therefore $-1\leq t_1,t_2,t_3\leq 1$. We have 
\[
\varepsilon'(c)=3c(1-t-c)(1-s-c)\geq 3c\cdot \delta/3\cdot \delta/3=c\delta^2/3>0.
\]
Hence the statement of the lemma holds.
\end{proof}

\begin{lm}\label{lm_1101up}
Take $\delta>0$. Then there exists an $\varepsilon>0$ such that 
\[
(s,s^2,s^3)+(t,t^2,t^3)+(-1,1,-1)+(0,0,\varepsilon')\in\A_{3,3}
\]
for all $-1<s<t<1$ and $0\leq\varepsilon'\leq\varepsilon$ with $1-t,t-s,s-(-1)\geq\delta$.
\end{lm}
\begin{proof}
The proof is similar to the proof of Lemma~\ref{lm_1110down}. For $0\leq\lambda\ll1$, one considers $(t_1,t_2,t_3)=(s-\lambda+\mu,t-\mu,-1+\lambda)$ where
\begin{align*}
\mu&=(\lambda+(t-s)-\sqrt{d})/2,\\
d&=(t-s)^2+2(t-s)\lambda+4(s-(-1))\lambda-3\lambda^2.
\end{align*}
Assuming that $d\geq0$ and $-1\leq t_1,t_2,t_3\leq1$, one can check that
\[
(s,s^2,s^3)+(t,t^2,t^3)+(-1,1,-1)+(0,0,\varepsilon'(\lambda))=\sum_{i=1}^3(t_i,t_i^2,t_i^3)\in\A_{3,3}
\]
for $\varepsilon'(\lambda)=3\lambda(t-(-1)-\lambda)(s-(-1)-\lambda)$. The other details are left to the reader.
\end{proof}

\begin{prop}\label{prop:C+C+1+-1}
Take $-1<s<t<1$. Then the point
\[
p=(s,s^2,s^3)+(t,t^2,t^3)+(1,1,1)+(-1,1,-1)
\] 
does not lie on the boundary of $\A_{3,4}$.
\end{prop}
\begin{proof}
Set $\delta=\min(1-t,t-s,s-(-1))/2>0$ and let $\varepsilon>0$ be the minimum of the two $\varepsilon$'s from Lemmas~\ref{lm_1110down} and~\ref{lm_1101up}. Write $p=(x,y,z)$. The Jacobian of the map 
\[
(u,v)\mapsto(u,u^2)+(v,v^2)+(1,1)+(-1,1)
\]
is invertible at $(s,t)$. It follows that all points in $\RR^2$ in a small neighborhood of $(x,y)$ are of the form
\[
(u,u^2)+(v,v^2)+(1,1)+(-1,1)
\]
with $(u,v)$ in a small neighborhood of $(s,t)$ by the Inverse Function Theorem. By shrinking this neighborhood, we may assume that $-1<u<v<1$ and $1-v,v-u,u-(-1)\geq\delta$. Lemmas~\ref{lm_1110down} and~\ref{lm_1101up} now tell us that
\[
(u,u^2,u^3)+(v,v^2,v^3)+(1,1,1)+(-1,1,-1)+(0,0,\varepsilon')\in\A_{3,4}
\]
for all $(u,v)$ in the neighborhood and $|\varepsilon'|<\varepsilon$. Hence $p$ is in the interior of $\A_{3,4}$.
\end{proof}

Finally, we take care of the cases $(k,\ell,a,b)=(1,2,1,0),(2,1,0,1)$.

\begin{lm}\label{lm_1200up}
Take $\delta>0$. Then there exists an $\varepsilon>0$ such that 
\[
(s,s^2,s^3)+2(t,t^2,t^3)+(0,0,\varepsilon')\in\A_{3,3}
\]
for all $-1\leq s<t<1$ and $0\leq\varepsilon'\leq\varepsilon$ with $1-t,t-s\geq\delta$.
\end{lm}
\begin{proof}
The proof is similar to the proof of Lemma~\ref{lm_1110down}. For $0\leq\lambda\ll1$, one considers
\[
(t_1,t_2,t_3)=(s+2\lambda,t-\lambda+\mu,t-\lambda-\mu)
\]
where $\mu=\sqrt{2(t-s)\lambda-3\lambda^2}$ and finds that
\[
(s,s^2,s^3)+2(t,t^2,t^3)+(0,0,\varepsilon'(\lambda))=\sum_{i=1}^3(t_i,t_i^2,t_i^3)\in\A_{3,3}.
\]
for $\varepsilon'(\lambda)=6(t-s)^2\lambda-24(t-s)\lambda^2+24\lambda^3$. The other details are left to the reader.
\end{proof}

\begin{lm}\label{lm_2100down}
Take $\delta>0$. Then there exists an $\varepsilon>0$ such that 
\[
2(s,s^2,s^3)+(t,t^2,t^3)-(0,0,\varepsilon')\in\A_{3,3}
\]
for all $-1<s<t\leq1$ and $0\leq\varepsilon'\leq\varepsilon$ with $t-s,s-(-1)\geq\delta$.
\end{lm}
\begin{proof}
The proof is similar to the proof of Lemma~\ref{lm_1110down}. For $0\leq\lambda\ll1$, one considers
\[
(t_1,t_2,t_3)=(s+\lambda+\mu,s+\lambda-\mu,t-2\lambda)
\]
where $\mu=\sqrt{2(t-s)\lambda-3\lambda^2}$ and finds that
\[
2(s,s^2,s^3)+(t,t^2,t^3)-(0,0,\varepsilon'(\lambda))=\sum_{i=1}^3(t_i,t_i^2,t_i^3)\in\A_{3,3}.
\]
for $\varepsilon'(\lambda)=6(t-s)^2\lambda-24(t-s)\lambda^2+24\lambda^3$. The other details are left to the reader.
\end{proof}

\begin{prop}\label{prop:C+2C+1}
Take $-1<s<t<1$. Then the point
\[
p=(s,s^2,s^3)+2(t,t^2,t^3)+(1,1,1)
\] 
does not lie on the boundary of $\A_{3,4}$.
\end{prop}
\begin{proof}
The proof is similar to the proof of Proposition~\ref{prop:C+C+1+-1}. For $(u,v)$ in a small neighborhood of $(s,t)$, we find points in $\A_{3,4}$ above
\[
(u,u^2,u^3)+2(v,v^2,v^3)+(1,1,1)
\]
using Lemma~\ref{lm_1200up} and we find points below using Lemma~\ref{lm_2100down}. Note here that in the latter case the role of the pair $(s,t)$ from Lemma~\ref{lm_2100down} is played by $(v,1)$.
\end{proof}

\begin{prop}\label{prop:2C+C-1}
Take $-1<s<t<1$. Then the point
\[
p=2(s,s^2,s^3)+(t,t^2,t^3)+(-1,1,-1)
\] 
does not lie on the boundary of $\A_{3,4}$.
\end{prop}
\begin{proof}
The proof is similar to the proof of Proposition~\ref{prop:C+2C+1}. For $(u,v)$ in a small neighborhood of $(s,t)$, we find points in $\A_{3,4}$ above 
\[
2(u,u^2,u^3)+(v,v^2,v^3)+(-1,1,-1)
\]
using Lemma~\ref{lm_1200up} and we find points below using Lemma~\ref{lm_2100down}. Note here that in the former case the role of the pair $(s,t)$ from Lemma~\ref{lm_1200up} is played by $(-1,u)$.
\end{proof}

By combining Propositions~\ref{prop_2par}, \ref{prop:2C+2C}, \ref{prop:C+C+1+-1}, \ref{prop:C+2C+1} and~\ref{prop:2C+C-1}, we see that the boundary of $\A_{3,n}$ is contained in $\B^{+}_{n}\cup\B^{-}_{n}$ for $n=4$. We now use this knowledge to prove the same for $n>4$. 

\begin{thm}\label{thm:boundaryA3nhalf}
The boundary of $\A_{3,n}$ is contained in the union of $\B^{+}_{n}$ and $\B^{-}_{n}$.
\end{thm}
\begin{proof}
For $n=3$, this follows directly from Proposition~\ref{prop_2par}. For $n=4$, we additionally use Propositions~\ref{prop:2C+2C}, \ref{prop:C+C+1+-1}, \ref{prop:C+2C+1} and~\ref{prop:2C+C-1}. For $n>4$, we need to show that points of the form
\[
k\begin{pmatrix}s\\s^2\\s^3\end{pmatrix}+\ell\begin{pmatrix}t\\t^2\\t^3\end{pmatrix}+a\begin{pmatrix}1\\1\\1\end{pmatrix}+b\begin{pmatrix}-1\\1\\-1\end{pmatrix}
\]
with $-1< s< t<1$, $k,\ell\geq1$ and $a,b\geq0$ are not on the boundary of $\A_{3,n}$ when one of the following conditions holds:
\begin{enumerate}
\item $k,\ell\geq2$,
\item $k=\ell=1$ and $a,b>0$,
\item $k=1$, $\ell>1$ and $a>0$,
\item $k>1$, $\ell=1$ and $b>0$.
\end{enumerate}
This is done by combining Lemma~\ref{lm:elim} with Propositions~\ref{prop:2C+2C}, \ref{prop:C+C+1+-1}, \ref{prop:C+2C+1} and~\ref{prop:2C+C-1}.
\end{proof}

\section{The semi-algebraic components of the boundary of \texorpdfstring{$\A_{3,n}$}{A3n}}\label{sect:kC+ellC}

Consider the sets
\[
\left\{k\begin{pmatrix}s\\s^2\\s^3\end{pmatrix}+\ell\begin{pmatrix}t\\t^2\\t^3\end{pmatrix}\,\middle|\,-1\leq s\leq t\leq1\right\}+a\begin{pmatrix}1\\1\\1\end{pmatrix}+b\begin{pmatrix}-1\\1\\-1\end{pmatrix}
\]
for $k,\ell\geq1$. Recall that
\begin{align*}
A_{k\ell} &= k\ell(k+\ell)^2,\\
B_{k\ell}(x,y) &= 2k\ell x(2x^2 - 3(k+\ell)y),\\
C_{k\ell}(x,y) &= x^6 - 3(k+\ell)x^4y + 3(k^2+k\ell+\ell^2)x^2y^2 - (k-\ell)^2(k+\ell)y^3,\\
D_{k\ell}(x,y) &= (k+\ell)y-x^2,\\
f_{k\ell}(x,y,z) &= A_{k\ell}z^2+B_{k\ell}(x,y)z+C_{k\ell}(x,y)
\end{align*}
and $B_{k\ell}^2-4A_{k\ell}C_{k\ell}=4k\ell(\ell-k)^2D_{k\ell}^3$ from Section~\ref{sect:mainres}. Our goal for this section is to prove the following proposition and theorem.

\begin{prop}\label{prop:f==g^2orirr}
If $k=\ell$, then 
\[
f_{k\ell}(x,y,z)=A_{k\ell}\left(z+\frac{B_{k\ell}(x,y)}{2A_{k\ell}}\right)^2
\]
decomposes the polynomial $f_{k\ell}$ into irreducible factors over $\QQ$. If $k\neq\ell$, then $f_{k\ell}$ is irreducible over $\CC$.
\end{prop}

\begin{thm}\label{thm:kC+ellC}
The set
\[
\left\{k\begin{pmatrix}s\\s^2\\s^3\end{pmatrix}+\ell\begin{pmatrix}t\\t^2\\t^3\end{pmatrix}\,\middle|\,-1\leq s\leq t\leq1\right\}
\]
consists of all points 
\[
(x,y,z)\in[-(k+\ell),(k+\ell)]\times[0,(k+\ell)]\times[-(k+\ell),(k+\ell)]
\]
such that $f_{k\ell}(x,y,z)=0$, the inequalities 
\[
0\leq k\ell D_{k\ell}(x,y)\leq \min\{k^2(k+\ell +x)^2,\ell^2(k+\ell-x)^2\}
\]
hold and in addition the following requirements are met:
\begin{itemize}
\item If $k<\ell$, then the inequality $z\leq \frac{-B_{k\ell}(x,y)}{2A_{k\ell}}$ must hold. 
\item If $k=\ell$, then equation $z=\frac{-B_{k\ell}(x,y)}{2A_{k\ell}}$ must hold. 
\item If $k>\ell$, then the inequality $z\geq \frac{-B_{k\ell}(x,y)}{2A_{k\ell}}$ must hold. 
\end{itemize}
\end{thm}

For the remainder of the section, we fix integers $k,\ell\geq1$ and we write 
\[
A=A_{k\ell}, B=B_{k\ell}, C=C_{k\ell}, D=D_{k\ell}, f=f_{k\ell}
\]
in order to simplify the used notation.

\begin{proof}[Proof of Proposition~\ref{prop:f==g^2orirr}]
The first statement is easy. Assume that $k\neq\ell$. To prove that $f$ is irreducible under this assumption, note that $f$ is homogeneous with respect to the grading where $\deg(x)=1$, $\deg(y)=2$ and $\deg(z)=3$. It follows that if $f$ is reducible, then 
\[
Az^2+B(x,y)z+C(x,y)=f=A(z+ax^3+bxy)(z+cx^3+dxy)
\]
for some $a,b,c,d$. However, this would imply that the coefficient 
\[
-(k-\ell)^2(k+\ell)
\]
of $C$ at $y^3$ equals $0$. This is a contradiction. So $f$ is irreducible.
\end{proof}

\begin{proof}[Proof of Theorem~\ref{thm:kC+ellC}]
Note that we have $x,z\in [-(k+\ell),(k+\ell)]$ and $y\in [0,(k+\ell)]$ for all points 
\[
\begin{pmatrix}x\\y\\z\end{pmatrix}\in\left\{k\begin{pmatrix}s\\s^2\\s^3\end{pmatrix}+\ell\begin{pmatrix}t\\t^2\\t^3\end{pmatrix}\,\middle|\,-1\leq s\leq t\leq1\right\}.
\]
So we let
\[
(x,y,z)\in[-(k+\ell),(k+\ell)]\times[0,(k+\ell)]\times[-(k+\ell),(k+\ell)]
\]
be a point and find out when it is contained in
\[
\left\{k\begin{pmatrix}s\\s^2\\s^3\end{pmatrix}+\ell\begin{pmatrix}t\\t^2\\t^3\end{pmatrix}\,\middle|\,-1\leq s\leq t\leq1\right\}.
\]
We start by looking at the first two coordinates. So we solve the system of equations
\begin{align*}
x&=ks+\ell t,\\
y&=ks^2+\ell t^2
\end{align*}
under the conditions that $-1\leq s\leq t\leq 1$. Solving the system, we find that
\[
(ks,\ell t)=\left(\frac{kx\pm\sqrt{k\ell D(x,y)}}{k+\ell},\frac{\ell x\mp\sqrt{k\ell D(x,y)}}{k+\ell}\right).
\]
So we need to assume that $k\ell D(x,y)\geq0$. Adding the condition $s\leq t$, we get
\[
(ks,\ell t)=\left(\frac{kx-\sqrt{k\ell D(x,y)}}{k+\ell},\frac{\ell x+\sqrt{k\ell D(x,y)}}{k+\ell}\right)
\]
and so the conditions $-1\leq s$ and $t\leq 1$ translate to
\[
\sqrt{k\ell D(x,y)}\leq k(k+\ell+x),\ell(k+\ell-x) .
\]
As $x\in[-(k+\ell),(k+\ell)]$, these conditions are equivalent to
\[
k\ell D(x,y)\leq k^2(k+\ell+x)^2,\ell^2(k+\ell-x)^2.
\]
Now, also consider the third coordinate $z=ks^3+\ell t^3$. One can check that $f(x,y,z)=0$. So if $k=\ell$, then we have
\[
z=\frac{-B(x,y)}{2A}
\]
by Proposition~\ref{prop:f==g^2orirr} and we are done. So assume that $k\neq \ell$. Then there are a priori two possibilities for $z$ given $x$ and $y$. However, given $s$ and $t$, it becomes clear that only one possibility remains. So we just need to find an inequality that selects the correct root of $f(x,y,-)$. One can check that
\[
k^2\ell^2(k+\ell)^3\left(z-\frac{-B(x,y)}{2A}\right)=(k^2-\ell^2)\sqrt{k\ell D(x,y)}^3.
\]  
So we find that  
\[
z\leq\frac{-B(x,y)}{2A}
\]
when $k<\ell$ and
\[
z\geq\frac{-B(x,y)}{2A}
\]
when $k>\ell$. This concludes the proof.
\end{proof}

\section{The sets \texorpdfstring{$\B^{+}_{n}$}{Bn+} and \texorpdfstring{$\B^{-}_{n}$}{Bn-}}\label{sect:B+-}

We are now ready to prove Theorem~\ref{thm:boundaryA3n} and the Main Theorem. Let $n\geq3$ be an integer. Recall the following notation from Section~\ref{sect:mainres}.
\begin{itemize}
\item We have
\begin{align*}
\C^{+}_{k,a}&=\left\{k\begin{pmatrix}s\\s^2\\s^3\end{pmatrix}+\begin{pmatrix}t\\t^2\\t^3\end{pmatrix}\,\middle|\,-1\leq s\leq t\leq 1\right\}+a\begin{pmatrix}1\\1\\1\end{pmatrix},\\
\C^{-}_{\ell,b}&=\left\{\begin{pmatrix}s\\s^2\\s^3\end{pmatrix}+\ell\begin{pmatrix}t\\t^2\\t^3\end{pmatrix}\,\middle|\,-1\leq s\leq t\leq 1\right\}+b\begin{pmatrix}-1\\1\\-1\end{pmatrix}
\end{align*}
for all integers $k,\ell\geq1$ and $a,b\geq0$.
\item We have $\B^{+}_{n}=\bigcup_{k=1}^{n-1}\C^{+}_{k,n-k-1}$ and $\B^{-}_{n}=\bigcup_{\ell=1}^{n-1}\C^{-}_{\ell,n-\ell-1}$.
\item The set $\B^{\flat}_{n}$ consists of all points $(x,y)\in\RR^2$ such that $ny\geq x^2$ and 
\[
y\leq n-1+(x+2i-(n-1))^2
\]
for each $i\in\{0,\dots,n-1\}$.
\item The projection map $\pi\colon\RR^3\to\RR^2$ sends $(x,y,z)\mapsto(x,y)$.
\end{itemize}
We start by listing some properties of $\B^{+}_{n}$ and $\B^{-}_{n}$. 

\begin{samepage}
\begin{prop}\label{prop:prop+}
Let $1\leq k\leq n-1$ be an integer.
\begin{itemize}
\item[(a)] The map
\begin{align*}
\mbox{\phantom{testtest}}\alpha_k\colon\{(s,t)\mid -1\leq s\leq t\leq 1\}&\to\pi(\C^{+}_{k,n-k-1}),\\
(s,t)&\mapsto k\begin{pmatrix}s\\s^2\end{pmatrix}+\begin{pmatrix}t\\t^2\end{pmatrix}+(n-k-1)\begin{pmatrix}1\\1\end{pmatrix},
\end{align*}
is a bijection.
\item[(b)] The boundary of $\pi(\C^{+}_{k,n-k-1})$ is the union of the following three sets:
\[
\mbox{\phantom{testtest}}\{\alpha_k(-1,t)\mid -1\leq t\leq 1\}, \{\alpha_k(s,s)\mid -1\leq s\leq 1\}, \{\alpha_k(s,1)\mid -1\leq s\leq 1\}. 
\]
\item[(c)] We have $\pi(\B^{+}_{n})=\B^{\flat}_{n}$.
\item[(d)] The projection map
\begin{align*}
\B^{+}_{n}&\to\B^{\flat}_{n},\\
(x,y,z)&\mapsto(x,y),
\end{align*}
is a bijection.
\end{itemize}
\end{prop}
\end{samepage}
\begin{proof}
To see (a), note that the map clearly is surjective. For injectivity, one has to solve $\alpha_k(s,t)=(x,y)$ for $s,t$ under the condition that $s\leq t$. This yields at most one solution for all $(x,y)$. For $(b)$, note that the Jacobian of the map $\alpha_k$ has full rank at all points $(s,t)$ with $-1<t<s<1$. From $(b)$ follows that the boundary of $\pi(\B^{+}_{n})$ is the union of
\[
\left\{n\begin{pmatrix}s\\s^2\end{pmatrix}\,\middle|\,-1\leq s\leq 1\right\}
\]
and
\[
\left\{i\begin{pmatrix}-1\\1\end{pmatrix}+\begin{pmatrix}t\\t^2\end{pmatrix}+(n-i-1)\begin{pmatrix}1\\1\end{pmatrix}\,\middle|\,-1\leq t\leq1\right\}
\]
for $i=0,\dots,n-1$. So the set itself is indeed given by the inequalities defining~$\B^{\flat}_{n}$. Finally, to see (d), it suffices to note that $\pi(\C^{+}_{k,n-k-1}\cap \C^{+}_{k+1,n-k})$ is equal to
\[
\left\{(k+1)\begin{pmatrix}s\\s^2\end{pmatrix}+(n-k-1)\begin{pmatrix}1\\1\end{pmatrix}\,\middle|\,-1\leq s\leq 1\right\}
\]
for $k=1,\dots,n-2$.
\end{proof}

\begin{prop}\label{prop:prop-}
Let $1\leq\ell\leq n-1$ be an integer.
\begin{itemize}
\item[(a)] The map
\begin{align*}
\mbox{\phantom{testtest}}\beta_\ell\colon\{(s,t)\mid -1\leq s\leq t\leq 1\}&\to\pi(\C^{-}_{\ell,n-\ell-1}),\\
(s,t)&\mapsto\begin{pmatrix}s\\s^2\end{pmatrix}+\ell\begin{pmatrix}t\\t^2\end{pmatrix}+(n-\ell-1)\begin{pmatrix}-1\\1\end{pmatrix},
\end{align*}
is a bijection.
\item[(b)] The boundary of $\pi(\C^{-}_{\ell,n-\ell-1})$ is the union of the following three sets:
\[
\mbox{\phantom{testtest}}\{\beta_\ell(-1,t)\mid -1\leq t\leq 1\}, \{\beta_\ell(t,t)\mid -1\leq t\leq 1\}, \{\beta_\ell(s,1)\mid -1\leq s\leq 1\}. 
\]
\item[(c)] We have $\pi(\B^{-}_{n})=\B^{\flat}_{n}$.
\item[(d)] The projection map
\begin{align*}
\B^{-}_{n}&\to\B^{\flat}_{n},\\
(x,y,z)&\mapsto(x,y),
\end{align*}
is a bijection.
\end{itemize}
\end{prop}
\begin{proof}
The proofs are similar to those of Proposition~\ref{prop:prop+}.
\end{proof}

The decomposition of $\B^{\flat}_n$ as a union of the projections of $\C^{+}_{1,n-2},\dots,\C^{+}_{n-1,0}$ is visualized in Figure~\ref{fig:projection+}. We note that the decomposition of $\B^{\flat}_n$ as a union of the projections of $\C^{-}_{1,n-2},\dots,\C^{-}_{n-1,0}$ looks similar but is mirrored along the vertical~axis.

\begin{figure}[ht]
  \centering
  \pgfplotsset{every axis/.append style={
      axis x line=middle,
      axis y line=left,
      axis line style={->,color=black},
      xlabel={$s$},
      ylabel={$t$},
  }}
  \begin{tikzpicture}
    \small
    \begin{axis}[
      declare function={
        alpha1(\n,\k,\s,\t) = \k*\s + \t + \n-\k-1;
        alpha2(\n,\k,\s,\t) = \k*\s^2 + \t^2 + \n-\k-1;
      },
      width=\textwidth,
      axis equal image,
      clip=false,
      xmin=-6,xmax=6,
      ymin=-0,ymax=6,
      xtick={-6,-4,0,4,6},
      xticklabels={$-n$,$-n+2$,$0$,$n-2$,$n$},
      ytick={0,1,5,6},
      yticklabels={$0$,$1$,$n-1$,$n$},
      hide axis,
      ]
      \addplot[domain=-1:1,samples=50]({alpha1(6,1,-1,x)}, {alpha2(6,1,-1,x)});
      \addplot[domain=-1:1,samples=50]({alpha1(6,2,-1,x)}, {alpha2(6,2,-1,x)});
      \addplot[domain=-1:-.4,samples=20]({alpha1(6,3,-1,x)}, {alpha2(6,3,-1,x)});
      \addplot[domain=-.4:.6,samples=20,dotted]({alpha1(6,3,-1,x)}, {alpha2(6,3,-1,x)});
      \addplot[domain=.6:1,samples=20]({alpha1(6,3,-1,x)}, {alpha2(6,3,-1,x)});
      \addplot[domain=-1:1,samples=50]({alpha1(6,4,-1,x)}, {alpha2(6,4,-1,x)})
        node[below right,pos=.4] {\scriptsize $\alpha_{n-2}(-1,s)$};
      \addplot[domain=-1:1,samples=50]({alpha1(6,5,-1,x)}, {alpha2(6,5,-1,x)})
        node[below right,pos=.4] {\scriptsize $\alpha_{n-1}(-1,s)$};
      \addplot[domain=-1:1,samples=50]({alpha1(6,1, x,1)}, {alpha2(6,1, x,1)});
      \addplot[domain=-1:1,samples=50]({alpha1(6,2, x,1)}, {alpha2(6,2, x,1)});
      \node at (axis cs:4,4.6) {$\pi(\C^{+}_{1,n-2})$};
      \addplot[domain=-1:1,samples=50]({alpha1(6,3, x,1)}, {alpha2(6,3, x,1)});
      \node at (axis cs:3,3.6) {$\pi(\C^{+}_{2,n-3})$};
      \addplot[domain=-1:1,samples=50]({alpha1(6,4, x,1)}, {alpha2(6,4, x,1)});
      \node at (axis cs:1.5,2.8) {$\iddots$};
      \addplot[domain=-1:1,samples=50]({alpha1(6,5, x,1)}, {alpha2(6,5, x,1)})
        node[above,pos=.25,rotate=-51] {\scriptsize $\alpha_{n-2}(s,s)$}
        node[below left,pos=.3,rotate=-51] {\scriptsize $\alpha_{n-1}(s,1)$};
      \node at (axis cs:0,2) {$\pi(\C^{+}_{n-2,1})$};
      \addplot[domain=-1:1,samples=50]({alpha1(6,5, x,x)}, {alpha2(6,5, x,x)})
        node[above,pos=.25,rotate=-50] {\scriptsize $\alpha_{n-1}(s,s)$};
      \node at (axis cs:-1.2,1) {$\pi(\C^{+}_{n-1,0})$};
    \end{axis}
  \end{tikzpicture}
  \caption{The set $\B^{\flat}_n =\pi(\C^{+}_{1,n-2})\cup\dots\cup\pi(\C^{+}_{n-1,0})$.}
  \label{fig:projection+}
\end{figure}

We can now prove Theorem~\ref{thm:boundaryA3n}.

\begin{proof}[Proof of Theorem~\ref{thm:boundaryA3n}]
We already know that (b) holds by Propositions~\ref{prop:prop+} and~\ref{prop:prop-}. We know that $\B^{+}_{n},\B^{-}_{n}\subseteq\A_{3,n}$, we know that the boundary of $\A_{3,n}$ is contained in $\B^{+}_{n}\cup\B^{-}_{n}$ by Theorem~\ref{thm:boundaryA3nhalf} and we know that the projection maps
\begin{align*}
  &\begin{aligned}
    \B^{+}_{n}&\to\B^{\flat}_{n},\\
    (x,y,z)&\mapsto(x,y),
  \end{aligned}
  &&\text{and}
  &\begin{aligned}
    \B^{-}_{n}&\to\B^{\flat}_{n},\\
    (x,y,z)&\mapsto(x,y),
  \end{aligned}
\end{align*}
are bijections by Propositions~\ref{prop:prop+} and~\ref{prop:prop-}. Together these statements imply (a). Let $(x,y)\in\B^{\flat}_{n}$ be a point. Then there exist unique numbers $z^{+},z^{-}\in\RR$ such that $(x,y,z^{+})\in\B^{+}_{n}$ and $(x,y,z^{-})\in\B^{-}_{n}$  by Propositions~\ref{prop:prop+} and~\ref{prop:prop-}. Our goal is to prove that $z^{+}\geq z^{-}$ with equality if and only if $(x,y)$ lies on the boundary of $\B_n^{\flat}$. Let $z\in\RR$ be a real number. From (a) and (b) it is clear that $(x,y,z)\in\A_{3,n}$ if and only if $z$ lies between $z^{+}$ and $z^{-}$. So $z^{+}=z^{-}$ when $(x,y)$ lies on the boundary of~$\B_n^{\flat}$. And, to prove that $z^{+}>z^{-}$ otherwise, it suffices to show that there exists a $z\in\RR$ such that $z<z^{+}$ and $(x,y,z)\in\A_{3,n}$. Now, let
\[
p=k\begin{pmatrix}s\\s^2\\s^3\end{pmatrix}+\begin{pmatrix}t\\t^2\\t^3\end{pmatrix}+(n-k-1)\begin{pmatrix}1\\1\\1\end{pmatrix}\in\C^{+}_{k,n-k-1}
\]
be a point where $-1\leq s\leq t\leq 1$ and recall Lemmas~\ref{lm_1110down} and~\ref{lm_2100down}. If $-1<s<t<1$ and $k=1$, then there is a point in $\A_{3,3}$ below 
\[
\begin{pmatrix}s\\s^2\\s^3\end{pmatrix}+\begin{pmatrix}t\\t^2\\t^3\end{pmatrix}+\begin{pmatrix}1\\1\\1\end{pmatrix}
\]
and hence a point in $\A_{3,n}$ below $p$. If $-1<s<t\leq1$ and $k\geq2$, then there is a point in $\A_{3,3}$ below 
\[
2\begin{pmatrix}s\\s^2\\s^3\end{pmatrix}+\begin{pmatrix}t\\t^2\\t^3\end{pmatrix}
\]
and hence a point in $\A_{3,n}$ below $p$. Taking into account how the sets $\C^{+}_{k,n-k-1}$ intersect, we find that there is a point in $\A_{3,n}$ below $p$ unless $s=-1$, $(k,t)=(1,1)$ or $(s,k)=(t,n-1)$, which are exactly the cases where $p$ projects to the boundary of~$\B_n^{\flat}$. This proves (c) and (d). Finally, using (a), (c), and Propositions~\ref{prop:prop+} and~\ref{prop:prop-}, we see that the boundary of $\A_{3,n}$ is the disjoint union of several (but not all) sets of the form
\[
\left\{k\begin{pmatrix}s\\s^2\\s^3\end{pmatrix}+\ell\begin{pmatrix}t\\t^2\\t^3\end{pmatrix}\,\middle|\,-1< s<t<1\right\}+a\begin{pmatrix}1\\1\\1\end{pmatrix}+b\begin{pmatrix}-1\\1\\-1\end{pmatrix}
\]
where $k,\ell,a,b\geq0$ have sum $n$. Given a point of the boundary, the number $s$ is unique when $k>0$ and the number $t$ is unique when $\ell>0$. Together with Proposition~\ref{prop_2par}, this shows (e). 
\end{proof}

Finally, we prove the Main Theorem.

\begin{proof}[Proof of the Main Theorem]
Fix a point $(x,y,z)\in\RR^3$. The following conditions are equivalent:
\begin{itemize}
\item[(a)] We have $(x,y)\in\B^{\flat}_n$.
\item[(b)] We have $(x,y)\in\pi(\C^{+}_{k,n-k-1})$ for some $k\in\{1,\dots,n-1\}$.
\item[(c)] We have $(x,y)\in\pi(\C^{-}_{\ell,n-\ell-1})$ for some $\ell\in\{1,\dots,n-1\}$.
\end{itemize}
Take $k,\ell\in\{1,\dots,n-1\}$. Then, using Proposition~\ref{prop:prop+}, we see that
\[
\pi(\C^{+}_{k,n-k-1})=\left\{\binom{x}{y}\in\RR^2\,\middle|\,\begin{array}{l}y\leq n-1+(x+k-(n-k-1))^2\\y\geq n-k-1+(k+1)^{-1}(x-(n-k-1))^2\\y\leq n-k+ k^{-1}(x-(n-k))^2\end{array}\right\}
\]
and we similarly get
\[
\pi(\C^{-}_{\ell,n-\ell-1})=\left\{\binom{x}{y}\in\RR^2\,\middle|\,\begin{array}{l}y\leq n-1+(x-\ell+(n-\ell-1))^2\\y\geq n-\ell-1+(\ell+1)^{-1}(x+(n-\ell-1))^2\\y\leq n-\ell+ \ell^{-1}(x+(n-\ell))^2\end{array}\right\}
\]
using Proposition~\ref{prop:prop-}. Assume that $(x,y)\in\B^{\flat}_n$ and that $k,\ell$ are as in (b) and (c). Using Theorem~\ref{thm:boundaryA3n}(d), we need to find conditions that express that $z^{-}\leq z\leq z^{+}$.
We have
\begin{align*}
f_{k1}(x-(n-k-1),y-(n-k-1),z^{+}-(n-k-1))&=0,\\
(n-k-1)+ \frac{-B_{k1}(x-(n-k-1),y-(n-k-1))}{2A_{k1}}&\leq z^{+}
\end{align*}
by Theorem~\ref{thm:kC+ellC}. So $z\leq z^{+}$ when 
\[
z\leq (n-k-1)+ \frac{-B_{k1}(x-(n-k-1),y-(n-k-1))}{2A_{k1}}=:\theta
\]
or
\[
f_{k1}(x-(n-k-1),y-(n-k-1),z-(n-k-1))\leq0.
\]
Note here that the polynomial $f_{k1}(x,y,-)$ has degree $2$ in $z$, that its leading coefficient is positive, that $z^{+}$ is its highest root and that it attains its minimum at $\theta$. This is visualized in Figure~\ref{fig:fk1}. We also have 
\begin{align*}
f_{1\ell}(x+(n-\ell-1),y-(n-\ell-1),z^{-}+(n-\ell-1))&=0,\\
-(n-\ell-1)+\frac{-B_{1\ell}(x+(n-\ell-1),y-(n-\ell-1))}{2A_{1\ell}}&\geq z^{-}
\end{align*}
by Theorem~\ref{thm:kC+ellC} and from this we conclude that $z\geq z^{-}$ if and only if
\[
z\geq -(n-\ell-1)+\frac{-B_{1\ell}(x+(n-\ell-1),y-(n-\ell-1))}{2A_{1\ell}}
\]
or
\[
f_{1\ell}(x+(n-\ell-1),y-(n-\ell-1),z+(n-\ell-1))\leq0.
\]
This leads to the semi-algebraic description of the Main Theorem.
\end{proof}

\begin{figure}[ht]
  \centering
  \pgfplotsset{every axis/.append style={
      axis x line=middle,
      axis y line=left,
      axis line style={->,very thin},
      xlabel={$z$},
  }}
  \begin{tikzpicture}
    \small
    \begin{axis}[
      every tick/.style={black,very thin},
      width=\axisdefaultwidth,
      height=.55*\axisdefaultheight,
      clip=false,
      xmin=-2.2,xmax=2.2,
      ymin=-2.05,ymax=2.2,
      major tick length=.3cm,
      xtick={0,1.414214},
      xticklabels={},
      ytick={0},
      ]
      \addplot[domain=-2:2,samples=50]({x},{x^2-2})
        node[above right,pos=0.07,xshift=-.2cm]
        {};
      \node[above,yshift=.1cm] at (axis cs:0,0) {$\theta$};
      \draw[dotted] (axis cs:0,0) -- (axis cs:0,-2);
      \node[above,yshift=.1cm] at (axis cs:{2^.5},0) (z+) {$z^{+}$};
      \draw[very thick] (axis cs:-2.3,0) -- (axis cs:0,0);
      \draw[very thick] (axis cs:{2^.5},0) -- (axis cs:0,0);
    \end{axis}
  \end{tikzpicture}
  \caption{Visualization of the condition $z\leq z^{+}$ in the proof of the Main Theorem. The parabola represents the function sending $z$ to $f_{k1}(x-(n-k-1),y-(n-k-1),z-(n-k-1))$.}
  \label{fig:fk1}
\end{figure}
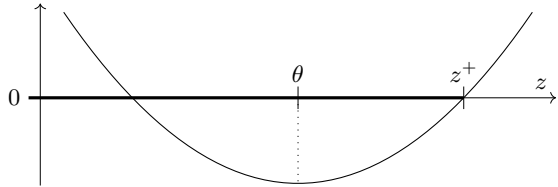

\section{Higher dimensions}\label{sect:future}

The Main Theorem provides a semi-algebraic description of the set $\A_{3,n}$ for each integer~$n\geq3$. So, a natural question to ask is: can we use the same proof strategy to find a semi-algebraic description of the sets $\A_{k,n}$ for $k>3$? At the moment, there still are some obstacles to doing so, which we discuss in this section.\bigskip

Following the same strategy as for $k=3$, we would again start by trying to find a description of the boundary of $\A_{k,n}$. One can check that the statement and proof of Proposition~\ref{prop_2par} carry over in a straightforward fashion for $k>3$, which yields a superset of the boundary. After this, one would again need to exclude points from this superset when they do not in fact lie on the boundary. In view of Theorem~\ref{thm:boundaryA3n}(d), proving that a point in $\A_{3,n}$ does not lie on the boundary can be done by showing that there are points in $\A_{3,n}$ above and below it. So an analogue of Theorem~\ref{thm:boundaryA3n}(d) for higher dimensions would be very useful. This leads to the following conjecture, which holds for $k\leq 3$.

\begin{conj}
Let $(x_1,\dots,x_k)$ be a point in $\A_{k,n}$. Then the set 
\[
\{y\in\RR\mid (x_1,\dots,x_{k-1},y)\in\A_{k,n}\}
\]
is a closed interval.
\end{conj}

Another approach might be through a generalization of Theorem~\ref{thm:boundaryA3n}(e). The boundary of $\A_{k,n}$ is contained in the union of the sets
\[
\left\{\ell_1\begin{pmatrix}t_1\\t_1^2\\\vdots\\t_1^k\end{pmatrix}+\dots+\ell_{k-1}\begin{pmatrix}t_{k-1}\\t_{k-1}^2\\\vdots\\t_{k-1}^k\end{pmatrix}\,\middle|\,-1<t_1<\dots< t_{k-1}<1\right\}+a\begin{pmatrix}1\\1\\\vdots\\1\end{pmatrix}+b\begin{pmatrix}-1\\(-1)^2\\\vdots\\(-1)^k\end{pmatrix}
\]
over all integers $\ell_1,\dots,\ell_{k-1},a,b\geq0$ that sum to $n$. The uniqueness of the representation of each point on the boundary would imply that the boundary of $\A_{k,n}$ is a disjoint union of some of these sets. It would also be a tool to eliminate some of these sets from consideration. This leads to our second conjecture, which also holds for $k\leq 3$.

\begin{conj}
Every point on the boundary of $\A_{k,n}$ can be written as 
\[
\begin{pmatrix}t_1\\t_1^2\\\vdots\\t_1^k\end{pmatrix}+\dots+\begin{pmatrix}t_n\\t_n^2\\\vdots\\t_n^k\end{pmatrix}
\]
for some tuple $(t_1,\dots,t_n)\in[-1,1]^n$. The set $\{t_1,\dots,t_n\}\setminus\{-1,1\}$ has at most $k-1$ elements and the tuple $(t_1,\dots,t_n)$ is unique up to permutation of its entries.
\end{conj}

Apart from finding the boundary of $\A_{k,n}$, there is also the problem of describing it semi-algebraically. When one would attempt this, the main obstacle to overcome is, in our opinion, finding an analogue of Theorem~\ref{thm:kC+ellC}. For $k=4$, this means we need to solve the following problem.

\begin{prb}\label{prb:a43}
Determine a semi-algebraic description of the set
\[
\left\{ \ell_1 \begin{pmatrix}t_1\\t_1^2\\t_1^3\\t_1^4\end{pmatrix}+ \ell_2 \begin{pmatrix}t_2\\t_2^2\\t_2^3\\t_2^4\end{pmatrix}+ \ell_3 \begin{pmatrix}t_3\\t_3^2\\t_3^3\\t_3^4\end{pmatrix}\,\middle|\, -1\le t_1\le t_2\le t_3\le 1 \right\}
\]
given the integers $\ell_1,\ell_2,\ell_3\ge 1$.
\end{prb}

These sets are expected to be the building blocks for the boundary of $\A_{4,n}$, so a solution to this problem seems essential if we want to apply the same approach we used for $\A_{3,n}$. Using elimination theory, we find that the Zariski closure of this set is a hypersurface defined by a single polynomial $f_{\ell_1,\ell_2,\ell_3}(x_1,x_2,x_3,x_4)$. This polynomial is homogeneous of degree $24$ with respect to the grading where $\deg(x_i)=i$ and has $169$ terms. Its coefficients are symmetric polynomials in $\ell_1,\ell_2,\ell_3$ of degree up to $18$.
When $\#\{\ell_1,\ell_2,\ell_3\}\leq 2$, the polynomial is a square. And, we have 
\[
f_{\ell,\ell,\ell}(x_1,x_2,x_3,x_4)=(-x_1^4 +6\ell x_1^2x_2 -3\ell^2 x_2^2 -\ell^2 8x_1x_3 +6\ell^3 x_4)^6.
\]
This suggests that we should first solve
\[
\begin{pmatrix}x_1\\x_2\\x_3\end{pmatrix}=\ell_1 \begin{pmatrix}t_1\\t_1^2\\t_1^3\end{pmatrix}+ \ell_2 \begin{pmatrix}t_2\\t_2^2\\t_2^3\end{pmatrix}+ \ell_3 \begin{pmatrix}t_3\\t_3^2\\t_3^3\end{pmatrix}
\]
for $t_1,t_2,t_3$ and then solve $f_{\ell_1,\ell_2,\ell_3}(x_1,x_2,x_3,x_4)=0$ for $x_4$. As this only involves solving polynomial equations of degree $\leq4$, this is theoretically doable. The problem however is to express the inequalities $-1\le t_1\le t_2\le t_3\le 1$ as polynomial inequalities in $x_1,x_2,x_3,x_4$.\bigskip

As an example, consider the case $\ell_1=\ell_2=\ell_3=1$. In this case, the set is contained in the hypersurface given by the equation
\[
x_1^4 - 6 x_1^2 x_2 + 3 x_2^2 + 8 x_1 x_3 - 6 x_4=0,
\]
which allows to eliminate the coordinate $x_4$. So here, the problem consists of finding a semi-algebraic description of the set
\[
\left\{\begin{pmatrix}x_1\\x_2\\x_3\end{pmatrix}=\begin{pmatrix}t_1\\t_1^2\\t_1^3\end{pmatrix}+\begin{pmatrix}t_2\\t_2^2\\t_2^3\end{pmatrix}+\begin{pmatrix}t_3\\t_3^2\\t_3^3\end{pmatrix}\,\middle|\, -1\le t_1\le t_2\le t_3\le 1 \right\}
\]
given $x_1,x_2,x_3\in\RR$.\bigskip

If we can solve Problem~\ref{prb:a43}, we still need to find analogues for the results in Section~\ref{sect:B+-}. These results rely on our complete understanding of the roots and extrema of parabolas. So to generalize these results, we probably need a similar level of understanding in the cases of cubics and quartics, which for now seems to be out of reach.


\begin{thebibliography}{99}

\bibitem{Newmansconj15}
A. Chang, D. Mehrle, S. J. Miller, T. Reiter, J. Stahl, D. Yott, 
\textit{Newman's conjecture in function fields}, 
J. Number Theory 157 (2015), pp. 154--169.

\bibitem{Sixshortchap}
Z.-L. Dou, Q. Zhang, 
\textit{Six Short Chapters on Automorphic Forms and $L$-functions}, 
Springer-Verlag Berlin Heidelberg (2012).

\bibitem{ZerodistrDir16}
P.-C. Hu, A.-D. Wu, 
\textit{Zero distribution of Dirichlet $L$-functions}, 
Ann. Acad. Sci. Fenn. Math. 41 (2016), pp. 775--788.

\bibitem{RHaTfP76}
G. L. Miller, 
\textit{Riemann's Hypothesis and Tests for Primality},
J. Comput. Syst. Sci. 13 (1976), no. 3, pp. 300--317.	 

\bibitem{Mi09}
S. J. Miller, 
\textit{An orthogonal test of the $L$-functions Ratios conjecture}, 
Proc. London Math. Soc. 99 (2009), no. 2, pp. 484--520.

\bibitem{UAnzdPrimRiem}
B. Riemann, 
\textit{Über die Anzahl der Primzahlen unter einer gegebenen Größe}, 
Monatsberichte der Berliner Akademie (November 1859).

\bibitem{RuSa_prep}
M. O. Rubinstein, P. Sarnak, 
\textit{The underdetermined matrix moment Problem I}, 
in preparation.

\end{thebibliography}
\end{document}